\documentclass[reqno,12pt]{amsart}

\usepackage{amssymb}
\usepackage{ifthen}

\newcommand{\showcomments}{no}    % leave this line uncomented *always*
\newcommand{\Z}{{\mathbb Z}}

\newcommand{\M}{{\mathsf M}}
\newcommand{\N}{{\mathsf N}}

\newcommand{\cG}{{\mathcal G}}

\newcommand{\oA}{{\overline A}}

\newcommand{\alp}{\alpha}
\newcommand{\bet}{\beta}
\newcommand{\gam}{\gamma}
\renewcommand{\phi}{\varphi}
\newcommand{\zet}{\zeta}

\theoremstyle{plain}
\newtheorem{lemma}{Lemma}
\newtheorem{proposition}{Proposition}
\newtheorem{theorem}{Theorem}

\theoremstyle{remark}
\newtheorem{example}{Example}

\newcommand{\refl}[1]{~\ref{l:#1}}
\newcommand{\refp}[1]{~\ref{p:#1}}
\newcommand{\refs}[1]{~\ref{s:#1}}
\newcommand{\reft}[1]{~\ref{t:#1}}
\newcommand{\refx}[1]{~\ref{x:#1}}
\newcommand{\refb}[1]{~\cite{b:#1}}
\newcommand{\refe}[1]{\eqref{e:#1}}

\newcommand{\seq}{\subseteq}
\newcommand{\stm}{\setminus}
\newcommand{\est}{\varnothing}

\renewcommand{\>}{\rangle}

\newcommand{\lfl}{\left\lfloor}
\newcommand{\rfl}{\right\rfloor}

\newcommand{\lfr}{\left\{}
\newcommand{\rfr}{\right\}}

\newcommand{\lpr}{\left(}
\newcommand{\rpr}{\right)}

\newcommand{\longc}{,\dotsc,}
\newcommand{\longp}{+\dotsb+}
\newcommand{\longu}{\cup\dotsb\cup}
\newcommand{\longop}{\oplus\dotsb\oplus}

\newcommand{\bt}{{\bf t}}

\newcommand{\Zn}[1][n]{\Z_5^{#1}}

\DeclareMathOperator{\rk}{rk}
\DeclareMathOperator{\diam}{diam}

\author{Vsevolod F. Lev}
\email{seva@math.haifa.ac.il}
\address{Department of Mathematics, The University of Haifa at Oranim,
  Tivon 36006, Israel}
\title[Stability result for sets with $3A\ne\Zn$]%
  {Stability result \\ for sets with $3A\ne\Zn$}


\begin{document}
\baselineskip = 16pt

\begin{abstract}
As an easy corollary of Kneser's Theorem, if $A$ is a subset of the
elementary abelian group $\Zn$ of density $5^{-n}|A|>0.4$, then $3A=\Zn$. We
establish the complementary stability result: if $5^{-n}|A|>0.3$ and
$3A\ne\Zn$, then $A$ is contained in a union of two cosets of an index-$5$
subgroup of $\Zn$. Here the density bound $0.3$ is sharp.

Our argument combines combinatorial reasoning with a somewhat non-standard
application of the character sum technique.
\end{abstract}

\maketitle

\section{Introduction}\label{s:intro}

For a subset $A$ of an (additively written) abelian group $G$, and a positive
integer $k$, denote by $kA$ the $k$-fold sumset of $A$:
  $$ kA := \{ a_1\longp a_k\colon a_1\longc a_k\in A \}. $$
How large can $A$ be given that $kA\ne G$? Assuming that $G$ is finite, let
  $$ \M_k(G) := \max \{ |A|\colon A\seq G,\ kA\ne G \}. $$
This quantity was introduced and completely determined by Bajnok in~\refb{b}.
The corresponding result, expressed in \refb{b} in a somewhat different
notation, can be easily restated in our present language.
\begin{theorem}[Bajnok {\cite[Theorem~6]{b:b}}]\label{t:bajnok}
For any finite abelian group $G$ and integer $k\ge 1$, writing $m:=|G|$, we
have
  $$ \M_k(G) = \max \lfr \lpr \lfl \frac{d-2}{k} \rfl +1 \rpr
                                         \frac{m}d \colon\, d\mid m \rfr $$
(where $\lfl\cdot\rfl$ is the floor function, and the maximum extends over
all divisors $d$ of $m$).
\end{theorem}

Once $\M_k(G)$ is known, it is natural to investigate the associated
stability problem: what is the structure of those $A\seq G$ with $kA\ne G$
and $|A|$ close to $\M_k(G)$?

There are two ``trivial'' ways to construct large subsets $A\seq G$
satisfying $kA\ne G$. One is to simply remove elements from a yet larger
subset with this property; another is to fix a subgroup $H<G$ and a set
$\oA\seq G/H$ with $k\oA\ne G/H$, and define $A\seq G$ to be the full inverse
image of $\oA$ under the canonical homomorphism $G\to G/H$. It is thus
natural to consider as ``primitive'' those subsets $A\seq G$ with $kA\ne G$
which are maximal subject to this property and, in addition, cannot be
obtained by the lifting procedure just described.

To proceed, we recall that the \emph{period} of a subset $A\seq G$, denoted
$\pi(A)$ below, is the subgroup consisting of all elements $g\in G$ such that
$A+g=A$:
  $$ \pi(A) := \{g\in G\colon A+g=A \}. $$
Alternatively, $\pi(A)$ can be defined as the (unique) maximal subgroup such
that $A$ is a union of its cosets. The set $A$ is called aperiodic if
$\pi(A)=\{0\}$, and periodic otherwise.

It is readily seen that a set $A\seq G$ with $kA\ne G$ can be obtained by
lifting if and only if it is periodic. Accordingly, motivated by the
discussion above, for a finite abelian group $G$ and integer $k\ge 1$, we
define $\N_k(G)$ to be the largest size of an aperiodic subset $A\seq G$
satisfying $kA\ne G$ and maximal under this condition:
\begin{multline*}
  \N_k(G) := \max \{ |A|\colon A\seq G,\ \pi(A)=\{0\},\ \\
           kA\ne G\ \text{and $k(A\cup\{g\})=G$ for each $g\in G\stm A$} \}
\end{multline*}
(subject to the agreement that $\max\est=0$). Clearly, we have $\N_k(A)\le
\M_k(A)$, and if the inequality is strict (which is often the case), then
determining $\N_k(G)$ is, in fact, a stability problem; for if $kA\ne G$ and
$|A|>\N_k(G)$, then $A$ is contained in the set obtained by lifting a subset
$\oA\seq G/H$ with $k\oA\ne G/H$, for a proper subgroup $H<G$.

The quantity $\N_k(G)$ is quite a bit subtler than $\M_k(G)$ and indeed, the
latter can be easily read off from the former; specifically, it is not
difficult to show that
  $$ \M_k(G) = \max \{ |H|\cdot \N_k(G/H) \colon H\le G \}. $$

\ifthenelse{\equal{\showcomments}{yes}}{%
\begin{proof}[Proof (to be suppressed in the ``canonical'' version)]
First, we show that $\M_k(G)$ does not exceed the expression in the RHS. This
is obvious if $\M_k(G)=0$, and we thus assume that $\M_k(G)>0$. Find in this
case $A\seq G$ with $|A|=\M_k(G)$ and $kA\ne G$, and let $H:=\pi(A)$. Then
$\phi_H(A)\seq G/H$ is aperiodic, satisfies $k\phi_H(A)\ne G/H$, and
$k(\phi_H(A)\cup(g+H))=G/H$ whenever $g+H\notin\phi_H(A)$. (For the last
property notice that $k(\phi_H(A)\cup(g+H))=G/H$ would imply
$k(A\cup(g+H))\ne G$ and, as a result, $\M_k(G)\ge|A\cup(g+H)|>|A|$.) It
follows that $\N_k(G/H)\ge|\phi_H(A)|=\M_k(G)/|H|$.

Conversely, assuming that the RHS is positive, fix $H\le G$ and $\oA\in G/H$
so that $|\oA|=\N_k(G/H)$ and $k\oA\ne G/H$, and let $A:=\phi_H^{-1}(\oA)$.
Then $kA\ne G$ implying $\M_k(G)\ge|A|=|\oA||H|=|H|\cdot \N_k(G/H)$. It
follows that $\M_k(G)$ is at least as large as the RHS.
\end{proof}}
{} % \showcomments=no  -- do nothing in this case

%more precisely, from the values of $\N_k$ on the quotient groups $G/H$.
%\begin{claim}\label{m:MN}
%For every finite abelian group $G$ and integer $k\ge 1$, we have
%  $$ \M_k(G) = \max \{ |H|\cdot \N_k(G/H) \colon H\le G \}. $$
%\end{claim}
%Claim~\refm{MN} will not be used below in this paper. We believe, however,
%that it can eventually turn useful elsewhere, and for this reason present the
%proof in the Appendix.

An invariant tightly related to $\N_k(G)$ was studied in \refb{kl}. To state
(the relevant part of) the results obtained there, following \refb{kl}, we
denote by $\diam^+(G)$ the smallest non-negative integer $k$ such that every
generating subset $A\seq G$ satisfies
 $\{0\}\cup A\cup\dotsb\cup kA=G$; that is, $k(A\cup\{0\})=G$. As shown in
\cite[Theorem~2.1]{b:kl}, if $G$ is of type $(m_1\longc m_r)$ with positive
integers $m_1\mid\dotsb\mid m_r$, then
\begin{equation}\label{e:diam}
  \diam^+(G) = \sum_{i=1}^r (m_i-1).
\end{equation}

\begin{theorem}[{\cite[Theorem~2.5 and Proposition 2.8]{b:kl}}]\label{t:kl2528}
For any finite abelian group $G$ and integer $k\ge 1$, we have
  $$ \N_k(G) \le \lfl\frac{|G|-2}k\rfl+1. $$
If $G$ is cyclic of order $|G|\ge k+2$ then, indeed, equality holds.
\end{theorem}

\begin{theorem}[{\cite[Theorem~2.4]{b:kl}}]\label{t:kl24}
For any finite abelian group $G$ and integer $k\ge 1$, denoting by $\rk(G)$
the smallest number of generators of $G$, we have
  $$ \N_k(G) = \begin{cases}
       |G|-1 \                &\text{if}\ k=1, \\
       \lfl\frac12\,|G|\rfl \ &\text{if}\ k=2<\diam^+(G), \\
       \rk(G)+1 \             &\text{if}\ k=\diam^+(G)-1, \\
       1 \             &\text{if}\ k\ge\diam^+(G)\ \text{and $|G|$ is prime}, \\
       0 \             &\text{if}\ k\ge\diam^+(G)\ \text{and $|G|$ is composite}.
             \end{cases} $$
\end{theorem}
% Theorem~\reft{kl24} establishes the values of $\N_k(G)$ for all finite
% abelian groups $G$, and all integer $k\in\{1,2\}$ and also
% s$k\ge\diam^+(G)-1$.

\begin{theorem}[{\cite[Theorem~2.7]{b:kl}}]\label{t:kl27}
For any finite abelian group $G$ with $\diam^+(G)\ge 4$, we have
  $$ \N_3(G) = \begin{cases}
                \frac13\,|G| \ &\text{if $3$ divides $|G|$}, \\
                \frac13\,(|G|-1) \ &\text{if every divisor of $|G|$
                                        is congruent to $1$ modulo $3$}.
              \end {cases} $$
\end{theorem}

In Section~\refs{kl}, we explain exactly how
Theorems~\reft{kl2528}--\!\reft{kl27} follow from the results of \refb{kl}.

Theorem~\reft{kl27} is easy to extend to show that, in fact, the equality
\begin{equation*}\label{e:cyclic3}
  \N_3(G) = \frac13\,(|G|-1)
\end{equation*}
holds true for any finite abelian group $G$ decomposable into a direct sum of
its cyclic subgroups of orders congruent to $1$ modulo $3$. Here the upper
bound is an immediate consequence of Theorem~\reft{kl2528}, while a
construction matching this bound is as follows.
\begin{example}\label{x:decomp}
Suppose that $G=G_1\longop G_n$, where $G_1\longc G_n\le G$ are cyclic with
$|G_i|\equiv 1\pmod 3$, for each $i\in[1,n]$. Write $|G_1|=3m+1$ and let
$H:=G_2\longop G_n$ so that $G=G_1\oplus H$. Assuming that
$\N_3(H)=\frac13(|H|-1)$, find an aperiodic subset $S\seq H$ with
$|S|=\frac13(|H|-1)$, such that $3S\ne H$ and $S$ is maximal subject to this
last condition. (If $n=1$ and $H$ is the trivial group, then take $S=\est$.)
Fix a generator $e\in G_1$, and consider the set
  $$ A := H \cup (e+H) \longu ((m-1)e+H) \cup (me+S) \seq G. $$
It is readily seen that $3A\ne G$ and $A$ is maximal with this property.
Furthermore,
  $$ |A| = m|H| + |S| = \frac13\,(|G|-1) $$
implying $\gcd(|A|,|G|)=1$, whence $A$ is aperiodic. As a result,
$\N_3(G)\ge|A|=\frac13(|G|-1)$.

Applying this construction recursively, we conclude that
$\N_3(G)\ge\frac13(|G|-1)$ whenever $G$ is a direct sum of its cyclic
subgroups of orders congruent to $1$ modulo $3$.
\end{example}

In contrast with Theorem~\reft{kl24} establishing the values of $\N_1(G)$ and
$\N_2(G)$ for all finite abelian groups $G$, Theorem~\reft{kl27} and the
remark following it address certain particular groups only, and it is by far
not obvious whether $\N_3(G)$ can be found explicitly in the general case. In
this situation it is interesting to investigate at least the most ``common''
families of groups not covered by Theorem~\reft{kl27} and
Example~\refx{decomp}, such as the homocyclic groups $\Z_m^n$ with
 $m\equiv 2\pmod 3$.

An important result of Davydov and Tombak \refb{dt}, well known for its
applications in coding theory and finite geometries, settles the problem for
the groups $\Z_2^n$; stated in our terms, it reads as
  $$ \N_3(\Z_2^n) = 2^{n-2}+1,\quad n\ge 4. $$
The goal of this paper is to resolve the next major open case, determining
the value of $\N_3(\Zn)$. To state our main result, we need two more
observations.

\begin{example}\label{x:11}
If $A\subset\Zn$ is a union of two cosets of a subgroup of index $5$, then
$3A\ne\Zn$, and $A$ is maximal with this property: that is,
$3(A\cup\{g\})=\Zn$ for every element $g\in\Zn\stm A$.
\end{example}
We omit the (straightforward) verification.

\begin{example}\label{x:22}
Let $n\ge 2$ be an integer. Fix a subgroup $H<\Zn$ of index $5$, an element
$e\in\Zn$ with $\Zn=H\oplus\<e\>$, and a set $S\seq H$ such that
$|S|=(|H|-1)/2$ and $0\notin 2S$. Finally, let
  $$ A := (H\stm\{0\}) \cup (e+S) \cup \{2e\}. $$
We have then $|A|=(3\cdot 5^{n-1}-1)/2$, and hence $A$ is aperiodic. Also, it
is easily verified that $3A=\Zn\stm\{4e\}$, and that $4e\in 3(A\cup\{g\})$
for any $g\in\Zn\stm A$.
\end{example}

The last example shows that
  $$ \N_3(\Zn) \ge \frac12\,(3\cdot 5^{n-1}-1),\quad n\ge 2. $$
With this estimate in view, we can eventually state the main result of our
paper.
\begin{theorem}\label{t:mainn}
Suppose that $n$ is a positive integer, and $A\seq\Zn$ satisfies
 $3A\ne\Zn$. If $|A|>3\cdot 5^{n-1}/2$, then $A$ is contained in a union
of two cosets of a subgroup of index~$5$. Consequently, in view of
Theorem~\reft{kl2528} and Example~\refx{22},
  $$ \N_3(\Zn) = \begin{cases}
         2\ &\text{if}\ n=1, \\
         \frac12\,(3\cdot 5^{n-1}-1)\ &\text{if}\ n\ge 2.
                  \end{cases}
  $$
\end{theorem}

We collect several basic results used in the proof of Theorem~\reft{mainn} in
the next section; the proof itself is presented in Section~\refs{proof}. In
Section~\refs{kl} we explain exactly how
Theorems~\reft{kl2528}--\!\reft{kl27} follow from the results of \refb{kl}.

In conclusion, we remark that any finite abelian group not addressed in
Example~\refx{decomp} has a direct-summand subgroup of order congruent to $2$
modulo $3$, and Example~\refx{22} generalizes onto ``most'' of such groups,
as follows.
\begin{example}
Suppose that the finite abelian group $G$ has a direct-summand subgroup
$G_1<G$ of order $|G_1|=3m+2$ with integer $m\ge 1$, and find a generator
$e\in G_1$ and a subgroup $H<G$ such that $G=G_1\oplus H$.

Assuming first that $|H|$ is odd, fix a subset $S\seq H$ with $0\notin 2S$
and $|S|=\frac12\,(|H|-1)$, and let
\begin{multline*}
  A := H \cup (e+H) \longu \big((m-2)e+H\big) \\
               \cup \big((m-1)e+(H\stm\{0\})\big) \cup (me+S) \cup \{(m+1)e\}.
\end{multline*}
A simple verification shows that $(3m+1)e\notin 3A$ and $A$ is maximal with
this property. Furthermore, since there is a unique $H$-coset containing
exactly $|H|-1$ elements of $A$, we have $\pi(A)\le H$, and since there is an
$H$-coset containing exactly one element of $A$, we actually have
$\pi(A)=\{0\}$. Therefore,
  $$ \N_3(G) \ge |A| = (m|H|-1) + |S| + 1
                                    = \frac{2m+1}{6m+4}\,|G|-\frac12. $$

Assuming now that $|H|$ is even, fix arbitrarily an element $g\in H$ not
representable in the form $g=2h$ with $h\in H$, find a subset $S\seq H$ with
$g\notin 2S$ and $|S|=\frac12|H|$, and let
\begin{multline*}
  A := H \cup (e+H) \longu \big((m-2)e+H\big) \\
               \cup \big((m-1)e+(H\stm\{g\})\big) \cup (me+S) \cup \{(m+1)e\}.
\end{multline*}
We have then $(3m+1)e+g\notin 3A$, and $A$ is maximal with this property.
Also, it is not difficult to see that $\pi(A)=\{0\}$. Hence,
  $$ \N_3(G) \ge |A| = (m|H|-1) + |S| + 1 = \frac{2m+1}{6m+4}\,|G|. $$
\end{example}

\section{Auxiliary Results}\label{s:aux}

For subsets $A$ and $B$ of an abelian group, we write $A+B:=\{a+b\colon a\in
A,\ b\in B\}$.

The following immediate corollary from the pigeonhole principle will be used
repeatedly.
\begin{lemma}\label{l:pigeon}
If $A$ and $B$ are subsets of a finite abelian group $G$ such that
 $A+B\ne G$, then $|A|+|B|\le|G|$.
\end{lemma}

An important tool utilized in our argument is the following result that we
will refer to below as \emph{Kneser's Theorem}.
\begin{theorem}[\cite{b:kn1,b:kn2}]\label{t:Kneser}
If $A$ and $B$ are finite subsets of an abelian group, then
  $$ |A+B| \ge |A| + |B| - |\pi(A+B)|. $$
\end{theorem}

Finally, we need the following lemma used in Kneser's original proof of his
theorem.
\begin{lemma}[\cite{b:kn1,b:kn2}]\label{l:Kneser}
If $A$ and $B$ are finite subsets of an abelian group, then
  $$ |A\cup B|+|\pi(A\cup B)| \ge \min \{ |A|+|\pi(A)|, |B|+|\pi(B)| \}. $$
\end{lemma}

\section{Proof of Theorem~\reft{mainn}}\label{s:proof}

We start with a series of results preparing the ground for the proof. Unless
explicitly indicated, at this stage we do not assume that $A$ satisfies the
assumptions of Theorem~\reft{mainn}.

For subsets $A,B\seq\Zn$ with $0<|B|<\infty$, by the \emph{density} of $A$ in
$B$ we mean the quotient $|A\cap B|/|B|$. In the case where $B=\Zn$, we speak
simply about the \emph{density of $A$}.

\begin{proposition}\label{p:1}
Let $n\ge 1$ be an integer, and suppose that $A\seq\Zn$ is a subset of
density larger than $0.3$. If $3A\ne\Zn$, then $A$ cannot have non-empty
intersections with exactly three cosets of an index-$5$ subgroup of $\Zn$.
\end{proposition}

\begin{proof}
Assuming that $3A\ne\Zn$ and  $F<\Zn$ is an index-$5$ subgroup such that $A$
intersects exactly three of its cosets, we obtain a contradiction.

Translating $A$ appropriately, we assume without loss of generality that
$0\notin 3A$. Fix $e\in\Zn$ such that $\Zn=F\oplus\<e\>$, and for $i\in[0,4]$
let $A_i:=(A-ie)\cap F$; thus,
$A=A_0\cup(e+A_1)\cup(2e+A_2)\cup(3e+A_3)\cup(4e+A_4)$ with exactly three of
the sets $A_i$ non-empty. Considering the action of the automorphisms of
$\Zn[]$ on its two-element subsets (equivalently, passing from $e$ to
$2e,3e$, or $4e$, if necessary), we further assume that one of the following
holds:
\begin{itemize}
\item[(i)]   $A_2=A_3=\est$;
\item[(ii)]  $A_3=A_4=\est$;
\item[(iii)] $A_0=A_4=\est$.
\end{itemize}
We consider these three cases separately.

\paragraph{Case (i): $A_2=A_3=\est$}
In this case we have $A=A_0\cup(e+A_1)\cup(4e+A_4)$, and from $0\notin 3A$ we
obtain $0\notin A_0+A_1+A_4$. Consequently, $|A_0|+|A_1+A_4|\le |F|$ by
Lemma~\refl{pigeon}, whence
\begin{align*}
  |A_0|+\max\{|A_1|,|A_4|\} &\le |F| \\
\intertext{and similarly,}
  |A_1|+\max\{|A_0|,|A_4|\} &\le |F|, \\
  |A_4|+\max\{|A_0|,|A_1|\} &\le |F|.
\end{align*}
Thus, denoting by $M$ the largest, and $m$ the second largest of the numbers
$|A_0|,|A_1|$, and $|A_4|$, we have $M+m\le|F|$. It follows that
  $$ |A| = |A_0|+|A_1|+|A_4| \le \frac32\,(M+m) \le \frac32\, |F|, $$
contradicting the density assumption $|A|>0.3\cdot 5^n$.

\paragraph{Case (ii): $A_3=A_4=\est$}
In this case from $0\notin 3A$ we get $3A_0\neq F$ and $A_1+2A_2\neq F$,
whence also $2A_0\ne F$ and $A_1+A_2\ne F$ and therefore $2|A_0|\le|F|$ and
$|A_1|+|A_2|\le|F|$ by Lemma~\refl{pigeon}. This yields
  $$ |A|=|A_0|+|A_1|+|A_2| \le \frac32\,|F|, $$
a contradiction as above.

\paragraph{Case (iii): $A_0=A_4=\est$}
Here we have $2A_1+A_3\neq F$ and $A_1+2A_2\neq F$ implying
$|A_1|+|A_3|\le|F|$ and $2|A_2|\le|F|$, respectively. This leads to a
contradiction as in Case (ii).
\end{proof}

\begin{lemma}\label{l:0.25}
Let $n\ge 1$ be an integer, and suppose that $A\seq\Zn$. If $2A$ has density
smaller than $0.5$, then $A$ has density smaller than $0.25$.
\end{lemma}

\begin{proof}
Write $H:=\pi(2A)$ and let $\phi_H\colon\Zn\to\Zn/H$ be the canonical
homomorphism. Applying Kneser's theorem to the set $A+H$ and observing that
$2(A+H)=2A+H=2A$, we get $|2A|\ge2|A+H|-|H|$, whence
$|\phi_H(2A)|\ge2|\phi_H(A)|-1$. If the density of $2A$ in $\Zn$ is smaller
than $0.5$, then so is the density of $\phi_H(2A)$ in $\Zn/H$ (in fact, the
two densities are equal); hence, in this case
  $$ \frac12\,|\Zn/H| > |\phi_H(2A)| \ge2|\phi_H(A)|-1. $$
This yields $|\phi_H(A)|<\frac14\big(|\Zn/H|+2\big)$ and thus, indeed,
$|\phi_H(A)|<\frac14\,|\Zn/H|$ as $|\Zn/H|\equiv 1\pmod 4$. It remains to
notice that the density of $A$ in $\Zn$ does not exceed the density of
$\phi_H(A)$ in $\Zn/H$.
\end{proof}

\begin{proposition}\label{p:2}
Let $n\ge 1$ be an integer, and suppose that $A\seq\Zn$ is a subset of
density larger than $0.3$, such that $3A\ne\Zn$. If $A$ has density larger
than $0.5$ in a coset of an index-$5$ subgroup $F<\Zn$, then $A$ has
non-empty intersections with at most three cosets of $F$.
\end{proposition}

\begin{proof}
Fix $e\in\Zn$ with $\Zn=F\oplus\<e\>$, and for $i\in[0,4]$ set
$A_i:=(A-ie)\cap F$; thus, $A=A_0\cup(e+A_1)\cup\dotsb\cup(4e+A_4)$. Having
$A$ replaced with its appropriate translate, we can assume that $A_0$ has
density larger than $0.5$ in $F$, whence $2A_0=F$ by Lemma~\refl{pigeon}. If
now $A_i$ is non-empty for some $i\in[1,4]$, then $ie+F=(ie+A_i)+2A_0\seq
3A$. This shows that at least one of the sets $A_i$ is empty. Moreover, we
can assume that \emph{exactly} one of them is empty, as otherwise the proof
is over. Replacing $e$ with one of $2e,3e$, or $4e$, is necessary, we assume
that $A_4=\est$ while $A_i\neq\est$ for $i\in[1,3]$, and aim to obtain a
contradiction. Notice, that
  $$ A=A_0\cup(e+A_1)\cup(2e+A_2)\cup(3e+A_3), $$
and that $ie+F\seq 3A$ for each $i\in[1,3]$ by the observation above,
implying $4e+F\nsubseteq 3A$. The last condition yields
\begin{equation}\label{e:loc1}
  A_0 + \big( (A_1+A_3) \cup 2A_2 \big) \neq F,
\end{equation}
and it follows from Lemma~\refl{pigeon} that
\begin{equation}\label{e:loc2}
  |A_0|+|(A_1+A_3) \cup 2A_2| \le |F|.
\end{equation}
Notice, that the last estimate implies $|2A_2|\le|F|-|A_0|<0.5|F|$, whence
\begin{equation}\label{e:A2small}
   |A_2|<0.25|F|
\end{equation}
by Lemma~\refl{0.25}.

Let $H$ be the period of the left-hand side of \refe{loc1}; thus, $H$ is a
proper subgroup of $F$, and we claim that, in fact,
\begin{equation}\label{e:loc3}
  |H|\le 5^{-2}|F|.
\end{equation}
To see this, suppose for a contradiction that $|F/H|=5$. Denote by $\phi_H$
the canonical homomorphism $\Zn\to\Zn/H$. From $|A_0|>0.5|F|$ we conclude
that $|\phi_H(A_0)|\ge 3$, and then \refe{loc1} along with
Lemma~\refl{pigeon} shows that
  $$ |\phi_H((A_1+A_3)\cup 2A_2)| \le 5-|\phi_H(A_0)| \le 2. $$
This gives $|\phi_H(A_2)|=1$, $\min\{|\phi_H(A_1)|,|\phi_H(A_3)|\}=1$, and
$\max\{|\phi_H(A_1)|,|\phi_H(A_3)|\}\le 5-|\phi_H(A_0)|$. As a result,
  $$ |\phi_H(A_0)|+|\phi_H(A_1)|+|\phi_H(A_2)|+|\phi_H(A_3)| \le 7, $$
implying $|A|=|A_0|+|A_1|+|A_2|+|A_3|\le 7|H|<1.5|F|$, contrary to the
density assumption. This proves~\refe{loc3}.

Since $\pi((A_1+A_3) \cup 2A_2)\le H$ by the definition of the subgroup $H$,
applying subsequently Lemma~\refl{Kneser} and then Kneser's theorem we obtain
\begin{align}\label{e:loc4}
  |(A_1+A_3) \cup 2A_2|
      &\ge \min \{ |A_1+A_3|+|\pi(A_1+A_3)|,\, |2A_2|+|\pi(2A_2)| \} - |H|
                \nonumber \\
      &\ge \min \{ |A_1|+|A_3|,\, 2|A_2| \} - |H|.
\end{align}
If $|A_1|+|A_3|\le 2|A_2|$, then from \refe{loc2}, \refe{loc4},
\refe{A2small}, and~\refe{loc3},
\begin{multline*}
  |F| \ge |A_0|+|A_1|+|A_3|-|H| = |A|-|A_2|-|H| \\
           > \frac32\,|F|-\frac14\,|F|-\frac1{25}|F| = \frac{121}{100}\,|F|,
\end{multline*}
a contradiction. Thus, we have
  $$ |A_1|+|A_3| > 2|A_2| $$
and then
  $$ |A_0|+2|A_2| \le |F|+|H| $$
by \refe{loc2} and \refe{loc4}. The latter estimate gives
  $$ \frac32\,|F| < |A| = |A_0|+|A_1|+|A_2|+|A_3|
                         \le \frac{|F|+|H|}2+\frac{|A_0|}2+|A_1|+|A_3|, $$
whence
  $$ \frac12\,|A_0| + |A_1| + |A_3| > |F|-\frac12|H|. $$
Using again \refe{loc2} and applying Kneser's theorem, we now obtain
\begin{multline*}
  |F| \ge |A_0|+|A_1+A_3| \ge |A_0|+|A_1|+|A_3|-|\pi(A_1+A_3)| \\
                       > \frac12\,|A_0| + |F| - \frac12\,|H| - |\pi(A_1+A_3)|
\end{multline*}
leading, in view of \refe{loc3}, to $|\pi(A_1+A_3)|\ge (|A_0|-|H|)/2>|F|/5$
and thus to $\pi(A_1+A_3)=F$. This, however, means that $A_1+A_3=F$,
contradicting \refe{loc1}.
\end{proof}

Propositions \refp{1} and \refp{2} show that to establish
Theorem~\reft{mainn}, it suffices to consider sets $A\seq\Zn$ with density
smaller than $0.5$ in every coset of every index-$5$ subgroup.

\begin{lemma}\label{l:ABC}
Let $n\ge 1$ be an integer, and suppose that $A,B,C\seq\Zn$ are subsets of
densities $\alp$, $\bet$, and $\gam$, respectively. If $0.4<\alp,\bet<0.5$
and $\alp+\bet+3\gam>1.5$, then $A+B+C=\Zn$.
\end{lemma}

\begin{proof}
Let $H:=\pi(A+B+C)$; assuming that $H\ne\Zn$, we obtain a contradiction. As
above, let $\phi_H\colon\Zn\to\Zn/H$ denote the canonical homomorphism.

If $|\Zn/H|=5$ then, in view of $|A|/|H|=5\alp>2$ we have $|\phi_H(A)|\ge 3$.
Similarly, $|\phi_H(B)|\ge 3$, and it follows that
$\phi_H(A)+\phi_H(B)=\Zn/H$; that is, $A+B+H=\Zn$. Hence,
$A+B+C=(A+B+H)+C=\Zn$, contradicting the assumption $H\ne\Zn$.

If $|\Zn/H|\ge 125$ then, by Kneser's Theorem and taking into account that
\begin{equation}\label{e:ABC}
  \pi(A+B)\le\pi(A+B+C)=H,
\end{equation}
we have
\begin{align*}
  |A+B+C| &\ge |A+B|+|C|-|H| \\
          &\ge |A|+|B|+|C|-2|H| \\
          &=   \frac23\,|A|+\frac23\,|B|+\frac13\big(|A|+|B|+3|C|)-2|H| \\
          &>   \Big(\frac23\cdot 0.4+\frac23\cdot0.4+\frac13\cdot1.5
                         -\frac2{125} \Big)\cdot 5^n \\
          &>   5^n,
\end{align*}
a contradiction.

Finally, consider the situation where $|\Zn/H|=25$. In this case
$|A|/|H|=25\alp>10$ whence $|A+H|\ge 11|H|$ and similarly, $|B+H|\ge 11|H|$.
In view of \refe{ABC}, Kneser's Theorem gives
  $$ |A+B+H| = |(A+H)+(B+H)| \ge |A+H|+|B+H|-|H| \ge 21|H|. $$
Also,
  $$ |C|/|H| =25 \gam > \frac{25}3(1.5-\alp-\bet) > \frac{25}6 > 4. $$
Consequently, $|C+H|\ge 5|H|$ and therefore
  $$ |A+B+H|+|C+H| \ge 26|H| > 5^n. $$
Lemma~\refl{pigeon} now implies $A+B+C=(A+B+H)+(C+H)=\Zn$, contrary to the
assumption $H\ne\Zn$.
\end{proof}

\begin{proposition}\label{p:3}
Let $n\ge 1$ be an integer, and suppose that $A\seq\Zn$ is a subset of
density larger than $0.3$, such that $3A\ne\Zn$. If $F<\Zn$ is an index-$5$
subgroup with the density of $A$ in every $F$-coset smaller than $0.5$, then
there is at most one $F$-coset where the density of $A$ is larger than $0.4$.
\end{proposition}

\begin{proof}
Suppose for a contradiction that there are two (or more) $F$-cosets
containing more than $0.4|F|$ elements of $A$ each. Shifting $A$ and choosing
$e\in\Zn\stm F$ appropriately, we can then write
$A=A_0\cup(e+A_1)\cup(2e+A_2)\cup(3e+A_3)\cup(4e+A_4)$ with
$A_0,A_1,A_2,A_3,A_4\seq F$ satisfying $\min\{|A_0|,|A_1|\}>0.4|F|$.

By Lemma~\refl{ABC} (applied to the group $F$), we have
  $$ 3A_0=2A_0+A_1=A_0+2A_1=3A_1=F, $$
implying $F\cup(e+F)\cup(2e+F)\cup(3e+F)\seq 3A$ and, consequently,
$4e+F\not\seq 3A$ by the assumption $3A\ne\Zn$. Furthermore, if we had
$2|A_0|+3|A_4|>1.5|F|$, this would imply $2A_0+A_4=F$ by Lemma~\refl{ABC},
resulting in $4e+F\seq 3A$; thus,
\begin{equation}\label{e:tmp62}
  2|A_0|+3|A_4|<1.5|F|.
\end{equation}
Similarly,
\begin{equation}\label{e:tmp63}
  |A_0|+|A_1|+3|A_3|<1.5|F|
\end{equation}
and
\begin{equation}\label{e:tmp64}
  2|A_1|+3|A_2|<1.5|F|
\end{equation}
(as otherwise by Lemma~\refl{ABC} we would have $A_0+A_1+A_3=F$ and
$2A_1+A_2=F$, respectively, resulting in $4e+F\seq 3A$). Adding
up~\refe{tmp62}--\refe{tmp64} we obtain
  $$ |A|=|A_0|+|A_1|+|A_2|+|A_3|+|A_4| < 1.5|F|=0.3\cdot 5^n, $$
contrary to the assumption on the density of $A$.
\end{proof}

We now use Fourier analysis to complete the argument and prove
Theorem~\reft{mainn}.

Suppose that $n\ge 2$, and that a set $A\seq\Zn$ has density $\alp>0.3$ and
satisfies $3A\ne\Zn$; we want to show that $A$ is contained in a union of two
cosets of an index-$5$ subgroup. Having translated $A$ appropriately, we can
assume that $0\notin 3A$. Denoting by $1_A$ the indicator function of $A$,
consider the Fourier coefficients
  $$ \hat1_A(\chi):=5^{-n}\sum_{a\in A} \chi(a),\ \chi\in\widehat{\Zn}. $$
For every character $\chi\in\widehat{\Zn}$, find a cube root of unity
$\zet(\chi)$ such that, letting $z(\chi):=-\hat1_A(\chi)\zet(\chi)$, we have
$\Re(z(\chi))\ge 0$. The assumption $0\notin 3A$ gives
  $$ \sum_\chi (\hat1_A(\chi))^3 = 0. $$
Consequently,
\begin{equation*}\label{e:Re}
  \sum_{\chi\ne1} \Re((z(\chi))^3)
                  = \Re\Big(\sum_{\chi\ne1}(-\hat1_A(\chi))^3\Big)=\alp^3,
\end{equation*}
and since $\Re(z)\ge 0$ implies $\Re(z^3)\le|z|^2\,\Re(z)$ (as one can easily
verify), it follows that
\begin{equation*}\label{e:Re}
  \sum_{\chi\ne1} |z(\chi)|^2\Re(z(\chi)) \ge \alp^3.
\end{equation*}
Comparing this to
  $$ \sum_{\chi\ne 1}|z(\chi)|^2=\alp(1-\alp) $$
(which is an immediate corollary of the Parseval identity), we conclude that
there exists a non-principal character $\chi$ such that
\begin{equation}\label{e:Rsmall}
  \Re(z(\chi)) \ge \frac{\alp^2}{1-\alp}.
\end{equation}
In view of $\alp>0.3$, it follows that
$\Re(-\hat1_A(\chi)\zet(\chi))>\frac9{70}$.

Replacing $\chi$ with the conjugate character, if needed, we can assume that
$\zet(\chi)=1$ or $\zet(\chi)=\exp(2\pi i/3)$. Let $F:=\ker\chi$, fix
$e\in\Zn$ with $\chi(e)=\exp(2\pi i/5)$, and for each $i\in[0,4]$, let
$\alp_i$ denote the density of $A-ie$ in $F$. By Propositions~\refp{1}
and~\refp{2},
% and in view of Example~\refx{11},
we can assume that $\max\{\alp_i\colon i\in[0,4]\}<0.5$, and then by
Proposition~\refp{3} we can assume that there is at most one index
$i\in[0,4]$ with $\alp_i>0.4$; that is, of the five conditions
 $\alp_i\le 0.4\ (i\in[0,4])$, at most one may fail to hold and must be
relaxed to $\alp_i<0.5$. We show that these assumptions are inconsistent
with~\refe{Rsmall}. To this end, we consider two cases.

\smallskip
\noindent
Case (i): $\zet(\chi)=1$. In this case we have
\begin{equation}\label{e:lopt1}
  \alp_0+\alp_1\cos(2\pi/5)\longp\alp_4\cos(8\pi/5)
                                         = 5 \Re(\hat1_A(\chi)) < -\frac9{14}.
\end{equation}
For each $k\in[0,4]$, considering $\alp_0\longc\alp_4$ as variables, we now
minimize the left-hand side of \refe{lopt1} under the constrains
\begin{gather}
  \alp_0\longp\alp_4\ge 1.5, \label{e:const1} \\
  \alp_k\in[0,0.5], \label{e:const2}
\intertext{and}
  \alp_i\in[0,0.4]\ \text{for all $i\in[0,4],\ i\ne k$}. \label{e:const3}
\end{gather}
This is a standard linear optimization problem which can be solved precisely,
and computations show that for every $k\in[0,4]$, the smallest possible value
of the expression under consideration exceeds $-9/14$. This rules out
Case~(i).

\smallskip
\noindent
Case (ii): $\zeta(\chi)=\exp(2\pi i/3)$. In this case we have
\begin{equation}\label{e:lopt2}
  \sum_{j=0}^4
           \alp_j\cos\lpr 2\pi \left( \frac13+\frac j5\right) \rpr
                       = 5 \Re(\hat1_A(\chi)\exp(2\pi i/3)) < -\frac9{14}.
\end{equation}
Minimizing the left-hand side of \refe{lopt2} under the
constrains~\refe{const1}--\refe{const3}, we see that its minimum is larger
than $-9/14$. This rules out Case~(ii), completing the proof of
Theorem~\reft{mainn}.

\section{From ${\bf t}_\rho^+(G)$ to $\N_k(G)$}\label{s:kl}

In Section~\refs{intro}, we mentioned the close relation between the quantity
$\N_k(G)$ and an invariant introduced in~\refb{kl}. Denoted by
$\bt^+_\rho(G)$ in \refb{kl}, this invariant was defined for integer $\rho\ge
1$ and a finite abelian group $G$ to be the largest size of an aperiodic
generating subset $A\seq G$ such that $(\rho-1)(A\cup\{0\})\ne G$ and $A$ is
maximal under this condition. It was shown in \refb{kl} that
$\bt^+_\rho(G)=0$ if $\rho>\diam^+(G)$, while otherwise $\bt^+_\rho(G)$ is
the largest size of an aperiodic subset $A\seq G$ satisfying
$(\rho-1)(A\cup\{0\})\ne G$ and maximal under this condition. Our goal in
this section is to prove the following simple lemma allowing one to
``translate'' the results of \refb{kl} into our present
Theorems~\reft{kl2528}--\!\reft{kl27}.
\begin{lemma}
For any finite abelian group $G$ and integer $k\ge 1$, we have
\begin{equation}\label{e:trans}
  \bt_{k+1}^+(G) = \N_k(G),
\end{equation}
except if $|G|$ is prime and $k\ge |G|-1$, in which case $\bt^+_{k+1}(G)=0$
and $\N_k(G)=1$.
\end{lemma}

\begin{proof}
We show that~\refe{trans} holds true unless $k\ge\diam^+(G)$ and $|G|$ is
prime; the rest follows easily.

Let $\cG$ denote the set of all aperiodic subsets $A\seq G$, and let $\cG_0$
be the set of all aperiodic subsets $A\seq G$ with $0\in A$.

Since translating a set $A\subseteq G$ affects neither its periodicity, nor
the property $kA=G$, we have
  $$ \N_k(G) = \max \{ |A|\colon A\in\cG_0,\ kA\neq G,
                     \ k(A\cup\{g\})=G\ \text{for each}\ g\in G\stm A \}. $$
As a trivial restatement,
\begin{multline}
  \N_k(G) = \max \{ |A|\colon A\in\cG_0,\ k(A\cup\{0\})\neq G, \label{e:Nk} \\
       \ k(A\cup\{0\}\cup\{g\})=G\ \text{for each}\ g\in G\stm A \}.
\end{multline}
However, letting $g=0$ shows that the conditions
  $$ k(A\cup\{0\})\ne G\ \text{and}
         \ k(A\cup\{0\}\cup\{g\})=G\ \text{for each}\ g\in G\stm A $$
automatically imply $0\in A$. Thus, in \refe{Nk}, the assumption $A\in\cG_0$
can be replaced with $A\in\cG$, meaning that $\N_k(G)$ is the largest size of
an aperiodic subset $A\seq G$ satisfying $k(A\cup\{0\})\ne G$ and maximal
under this condition; consequently, taking into account the discussion at the
beginning of this section, if $k<\diam^+(G)$, then $\N_k(G)=\bt^+_{k+1}(G)$.

Consider now the situation where $k\ge\diam^+(G)$. In this case
$\bt^+_{k+1}(G)=0$, and by the definition of $\diam^+(G)$, for any generating
subset $A\seq G$ we have $k(A\cup\{0\})=G$. Suppose that $A\in\cG$ satisfies
$kA\ne G$ and is maximal subject to this condition. (If such sets do not
exist, then $\N_k(G)=0=\bt^+_{k+1}(G)$.) Translating $A$ appropriately, we
can assume that $0\in A$, and then $k(A\cup\{0\})=kA\ne G$. It follows that
$A$ is not generating; that is, $H:=\<A\>$ is a proper subgroup of $G$.
Furthermore, the maximality of $A$ shows that $A=H$ is a maximal subgroup,
and aperiodicity of $A$ gives $A=H=\{0\}$. Therefore $G$ has prime order.
\end{proof}

\bigskip


\vfill
\begin{thebibliography}{DT89}
\bibitem[B15]{b:b}
  {\sc B.~Bajnok},
  The $h$-critical number of finite abelian groups,
  \emph{Uniform Distribution Theory} {\bf 10} (2015), no.~2, 93--115.
\bibitem[DT89]{b:dt}
  {\sc A.A.~Davydov} and {\sc L.M.~Tombak},
  Quasiperfect linear binary codes with distance $4$ and complete caps in
    projective geometry (Russian),
  \emph{Problemy Peredachi Informatsii} {\bf 25} (1989), no. 4, 11--23;
    translation in \emph{Problems Inform. Transmission} {\bf 25} (1989), no. 4,
    265--275 (1990).
\bibitem[KL09]{b:kl}
  {\sc B.~Klopsch} and {\sc V.F.~Lev},
  Generating abelian groups by addition only,
  \emph{Forum Mathematicum} {\bf 21} (2009), no. 1, 23--41.
\bibitem[Kn53]{b:kn1} {\sc M.~Kneser},
  Absch\"atzung der asymptotischen Dichte von Summenmengen,
  \emph{Math. Z.} {\bf 58} (1953), 459--484.
\bibitem[Kn55]{b:kn2} \bysame,
  Ein Satz \"uber abelsche Gruppen mit Anwendungen auf die
  Geometrie der Zahlen,
  \emph{Math. Z.} {\bf 61} (1955), 429--434.
\end{thebibliography}
\end{document}